\documentclass{amsart}
\usepackage{amssymb}
\usepackage{graphicx}
\newcommand{\PP}{\mathbb{P}}
\newcommand{\CC}{\mbox{${\mathbb C}$}}
\newcommand{\F}{\mbox{${\mathcal F}$}}

\newcommand{\I}{\mbox{${\mathcal I}$}}
\newcommand{\ZZ}{\mbox{${\mathbb Z}$}}
\newcommand{\Oc}{\mbox{${\mathcal O}$}}

\newcommand{\lra}{\longrightarrow}

\newtheorem{lema}{Lemma}[section]
\newtheorem{corolario}[lema]{Corollary}
\newtheorem{teorema}[lema]{Theorem}

\newtheorem{proposicion}[lema]{Proposition}
\newtheorem{remark}[lema]{Remark}

\newtheorem{definicion}[lema]{Definition}

\newenvironment{demostracion}
  {\noindent {{\it  Proof.}}}
  {\par \hfill \fbox{}}

\begin{document}

\title[{\sc On the singular scheme of projective foliations in $\PP^3$}]
{On the singular scheme of codimension one holomorphic foliations
in $\PP^3$}
\author{Luis Giraldo}
\address{Departamento de Geometr\'{\i}a y Topolog\'{\i}a, Facultad
de Ciencias Matem\'aticas, Universidad Complutense de Madrid.
Plaza de Ciencias 3, Ciudad Universitaria. 28040 Madrid, SPAIN}
\email{luis.giraldo@mat.ucm.es}
\author{Antonio J. Pan-Collantes}
\address{Departamento de Matem\'aticas, Facultad de Ciencias, Universidad
de C\'adiz. Pol\'{\i}gono R\'{\i}o San Pedro. 11510 Puerto Real
(C\'adiz), SPAIN.} \email{antonio.pan@uca.es}

\thanks{Partially supported by Plan Nacional I+D grant no.
MTM2004-07203-C02-02 and MTM2007-61124, Spain. The second author
also supported by the FPU grant no. AP2006-03911,   Ministerio de
Educaci\'on y Ciencia, Spain}
\date{September 2008}


\begin{abstract}
In this work, we begin by showing that a holomorphic foliation
with singularities is reduced if and only if its normal sheaf is
torsion free. In addition, when the codimension of the singular
locus is at least two, it is shown that being reduced is
equivalent to the reflexivity of the tangent sheaf. Our main
results state on one hand, that the tangent sheaf of a codimension
one foliation in $\PP^3$ is locally free if and only the singular
scheme is a curve, and that it splits if and only if that curve is
arithmetically Cohen-Macaulay. On the other hand, we discuss when
a split foliation in $\PP^3$ is determined by its singular scheme.
\end{abstract}

\maketitle

\section{Introduction}

The study of holomorphic foliations with singularities in complex manifolds has attracted a lot of
interest over the last 40 years. There are very good sources to learn about the different aspects
of the theory, and we just mention here the classical paper \cite{Baum-Bott} where a general theory
in terms of coherent sheaves was developed, and the recent book \cite{Suwa}.

\vskip 2mm

The tangent sheaf of the foliation, its normal sheaf and singular
locus are the key elements for this study. Needless to say, a very
important problem  is to analyze if the tangent sheaf of the
foliation is locally free as well as further properties (as the
splitting) when it is.
These properties have important consequences, as for example the ones
recently noted in \cite{Nosotros} and
\cite{Cukierman-Pereira}, where it has been proved that the fact that the
tangent sheaf splits, along with some properties of the singular
locus, give stability under deformations of the foliation and make
it possible to characterize certain components of the space that
parameterizes holomorphic foliations.

\vskip 2mm

The first section of this paper is devoted to the study of the
tangent and normal sheaves. We prove some key technical results
for the main goals of our research that, even when mentioned in
some references, were not explicitly found in the literature on
singular foliations. We think that they could be useful for other
people working in the field. Namely, we prove that a foliation is
reduced if and only if  its normal sheaf is torsion free and that,
under the assumption that the singular locus is of codimension at
least two, a foliation is reduced if and only if the tangent sheaf
is reflexive.

\vskip 2mm

As an immediate application  of  the reflexivity of the tangent sheaf, we give an affirmative
answer to the question posed by Suwa in \cite{Suwa}: Is  the
tangent sheaf to a 1 dimensional reduced foliation locally free?

\vskip 2mm

The main results contained in this paper establish a link between the above mentioned properties
of the tangent sheaf and some algebro-geometric properties of the singular scheme. For codimension
one foliations in $\PP^3$, using some results by Roggero on reflexive sheaves (see \cite{Roggero-Valabrega})
we prove that the
tangent sheaf is locally free if and only if the singular scheme is a curve, i.e. a Cohen-Macaulay scheme of
pure dimension one (Theorem \ref{locallyfree}); moreover, this sheaf splits if and only if the singular
scheme is an arithmetically Cohen-Macaulay curve (Theorem \ref{splitaCM}).

We also prove (Theorem \ref{splitsaturado}) an algebraic characterization of the splitting of a locally  free tangent
sheaf. It is given by the fact that
the homogeneous ideal generated by the coefficients of a form defining the foliation, that defines the
singular scheme, is saturated.

\vskip 0.5cm

An interesting problem is to decide when the singular scheme uniquely determines the foliation. That is,
given a foliation defined by a 1-form $\omega$, and having $Z$ as singular scheme, is it true
that if the 1-form $\omega^{\prime}$ defines a foliation with $Z$ as its singular scheme then
we have $\omega=\lambda \omega^{\prime}$ for $\lambda \in \CC^*$?.

In dimension 2, the problem was completely solved recently by Campillo and Olivares in
\cite{Campillo-Olivares}: for reduced foliations of degree $d\neq 1$ there are no two different foliations with the same
singular scheme.

\vskip 0.3cm

With our approach, we get simpler proofs of their main results,
and more importantly we prove some results in the $\PP^3$ case. Namely, we prove
unique determination by the singular
scheme for codimension one degree $d$ foliations in $\PP^3$ with
split tangent sheaf, when the splitting type is $\mathcal
O_{\PP^3}(a) \oplus \mathcal O_{\PP^3}(b)$, with $a, b\leq -1$.

We also show that if there is a subbundle of the tangent bundle of the foliation of the form
$\mathcal O_{\PP^3}(1)$, then it
is a linear pull-back of a foliation in a plane, and so is determined by its singular scheme if $d\neq 1$.

When there is a $\mathcal O_{\PP^3}$ subbundle of
$\mathcal F$, we just show some examples that illustrate the situation: there are foliations determined by the singular
scheme (exceptional foliations, see \cite{Nosotros}), and others that are not: the logarithmic foliations
of type $\mathcal L (1,1,1,1)$.

\section{Basic facts concerning the tangent sheaf}

\subsection{Preliminaries}

In this section we recall the basics of the theory. The following definition is taken from \cite{Baum-Bott}:
\begin{definicion}
Let $M$ be a complex $n$-dimensional manifold, and let $TM$ be its tangent sheaf.
\vskip 2mm

\noindent $\bullet$ A codimension $k$ holomorphic foliation with singularities in $M$
is an injective morphism of sheaves $\varphi: {\mathcal F} \longrightarrow TM$,
such that $\varphi({\mathcal F})$ is an integrable coherent subsheaf of $TM$, of rank $n-k$.
The integrability means that for each point
$x\in M$, $\varphi({\mathcal F})_x$ is closed under the bracket operation for vector fields.

\noindent {\rm Notation: we will denote by ${\mathcal F}_{\varphi}$ the foliation given by
$\varphi: {\mathcal F} \longrightarrow TM$}.

\noindent $\bullet$ The sheaves ${\mathcal F}$ and the  ${\mathcal N_{\mathcal F_{\varphi}}}:=TM/\varphi({\mathcal F})$ are,
respectively, the tangent and normal sheaves
of the foliation.

\noindent $\bullet$
$
Sing({\mathcal F}_{\varphi})=\{ x\in M \, | \, (\mathcal N_{\mathcal F_{\varphi}})_x\,
\text{is not a free}\,{\mathcal O}_{M,x}-\text{module}\}
$
is the singular set of the foliation ${\mathcal F}_{\varphi}$.
\end{definicion}

The following facts can be found, for instance,  in  \cite{Okonek}:
\begin{enumerate}
\item The normal sheaf is coherent.
\item By definition, $Sing({\mathcal F}_{\varphi})$ is the singular set of the sheaf ${\mathcal N_{\mathcal F_{\varphi}}}$.
It is a closed analytic subvariety of $M$ of codimension at least one.
\item There is a rank $n-k$ holomorphic vector bundle $F$ on $M\setminus Sing({\mathcal F}_{\varphi})$ whose sheaf of
sections is $\varphi({\mathcal F})|_{M\setminus Sing({\mathcal
F}_{\varphi})}$.
\end{enumerate}

Note also that being a subsheaf of a locally free sheaf, ${\mathcal F}$ is  torsion free
and therefore its singular set
$S({\mathcal F})=\{x\in M\, | \, {\mathcal F}_x \,\text{is not a free}\, {\mathcal O}_{M,x}-\text{module}\}$
is of codimension $\geq 2$.

Moreover, it is also clear that $S({\mathcal F})\subset Sing({\mathcal F}_{\varphi})$.
Indeed, if $x\not\in Sing({\mathcal F}_{\varphi})$ there is a neighbourhood $V$ of $x$ such that the sequence
$$
0\longrightarrow {\mathcal F}(V) \longrightarrow TM(V) \longrightarrow {\mathcal N_{\mathcal F_{\varphi}}}(V)
\longrightarrow 0
$$
is an exact sequence of ${\mathcal O}_{M}(V)$ modules, and the second and third modules are free. Hence
(see \cite{Eisenbud}, Appendix 3)
the sequence splits since ${\mathcal N_{\mathcal F_{\varphi}}}(V)$ is projective and ${\mathcal F}(V)$ is a direct summand
of a free module, and so it is projective.
We conclude that, for every $p\in V$, ${\mathcal F}_p$ is a free ${\mathcal O}_{M,p}$ module. In particular,
$x\not\in S({\mathcal F})$.

\vskip 3mm

In the study of singular holomorphic foliations, it is usual to deal with reduced ones. Let us recall their definition
(\cite{Baum-Bott}):
\begin{definicion} Let ${\mathcal F}_{\varphi}$  be a foliation:

\noindent $\bullet$ $\varphi({\mathcal F})$ is full if given any open set
$U\subset M$, and $\gamma $ a holomorphic section of $TM|_U$ such that
$\gamma (x)\in \varphi({\mathcal F})_x$ for each $x\in U\cap (M\setminus Sing({\mathcal F}_{\varphi}))$,
then at each point $p\in U\cap Sing({\mathcal F}_{\varphi})$ the germ of the holomorphic vector field $\gamma$ is in
$\varphi({\mathcal F})_p$.

\noindent $\bullet$ ${\mathcal F}_{\varphi}$ is reduced if $\varphi({\mathcal F})$ is full.
\end{definicion}

\vskip 3mm

\subsection{Some remarks}

For the sake of completeness and lack of reference, we now state and prove some properties
of the tangent and normal sheaves of a holomorphic foliation.
Recall that a coherent sheaf $\F$ is reflexive if and only if it is isomorphic to its bidual $\F^{**}$ (\cite{Okonek},
\cite{Hartshorne}).
\begin{remark}\label{main}
The foliation ${\mathcal F}_{\varphi}$ is reduced if and only if
${\mathcal N_{\mathcal F_{\varphi}}}$ is torsion free.
If the singular locus is of codimension at least $2$, the foliation is  reduced if
and only if ${\mathcal F}$ is a reflexive sheaf.
\end{remark}

Indeed,
assume that ${\mathcal F}_{\varphi}$ is reduced, hence  $\varphi({\mathcal F})$ is full. Torsion freeness is a local
property and locally free sheaves are torsion free, so it is enough to show that $(\mathcal N_{\mathcal
F_{\varphi}})_x$ is a torsion free ${\mathcal O}_{M,x}$ -module for $x\in Sing({\mathcal F}_{\varphi})$.

Let us suppose, on the contrary,  that it is not. Then, there is a nonzero element  $m\in (\mathcal N_{\mathcal
F_{\varphi}})_x$ and a nonzero element $a\in {\mathcal O}_{M,x}$ such that $am=0$ in $(\mathcal N_{\mathcal
F_{\varphi}})_x$. Let $U$ be a small enough neighborhood of $x$ in $M$, so that we can take $\alpha \in {\mathcal
O}_{M}(U)$ and $\gamma \in TM(U)$ such that $\alpha_x=a$, $\pi_x(\gamma_x)=m$.

${\mathcal N_{\mathcal F_{\varphi}}}$ is locally free in $V=U\cap (M\setminus Sing({\mathcal
F}_{\varphi}))$, hence torsion free and we have that
$$
\gamma (U\cap (M\setminus Sing({\mathcal F}_{\varphi})) \in
\varphi({\mathcal F})(U\cap (M\setminus Sing({\mathcal F}_{\varphi})).
$$
As  $\varphi({\mathcal F})$ is full, we have that $\gamma (U)\in \varphi({\mathcal F})(U)$,
which gives $m=0$ in $(\mathcal N_{\mathcal F_{\varphi}})_x$, a
contradiction.

\vskip 3mm

Conversely, suppose that ${\mathcal N_{\mathcal F_{\varphi}}}$ is torsion free. If $U\subset M$ is an open subset and
$\gamma $
is a section of $TM(U)$ such that $\gamma_x\in (\varphi({\mathcal F}))_x$
for each $x\in U\cap (M\setminus Sing({\mathcal F}_{\varphi}))$,
then for a nonzero element $f\in {\mathcal O}_M(U)$  vanishing on the
analytic set $Sing({\mathcal F}_{\varphi})$, and for each $p\in U\cap Sing({\mathcal F}_{\varphi})$,
 $(f\,\gamma)_p=f_p {\gamma}_p$ gives the zero element in $({\mathcal N_{\mathcal F_{\varphi}}})_p$. Hence,
$\gamma_p \in (\varphi({\mathcal F}))_p$. Thus, $\varphi({\mathcal F})$
is full and the foliation is reduced.

\vskip 2mm

The second part of the statement follows from the first,  using
that a torsion free sheaf $\mathcal F$ is reflexive if and only if there is a locally free sheaf
$\mathcal E$ such that $\mathcal F\subset \mathcal E$ and $\mathcal E / \mathcal F$ is torsion free
(see \cite[p.61]{miyaoka}).

\vskip 0.25cm

\noindent We can make the following further remarks:
\begin{itemize}
\item The singular locus of a reflexive sheaf is  of codimension $\geq 3$ (see \cite{Hartshorne}, \cite{Okonek}). Hence, the
same is true for the singular set of the tangent sheaf of a reduced holomorphic foliation.
\item
The tangent sheaf of a foliation ${\mathcal F}$ is torsion free and we have (\cite{Okonek}), that ${\mathcal F}\subset
{\mathcal F}^{**}$. Hence, if the singular locus of the foliation is of codimension $\geq 2$ we can consider the
reduced foliation given by the double dual of ${\mathcal F}$.
\end{itemize}

\vskip 0.3cm

Let us finally obtain a consequence of Remark \ref{main}.
Suwa, in \cite{Suwa}, made a distinction between  1-dimensional singular holomorphic foliations
and ``foliations by
curves'', a concept that was defined in \cite{Gomez-Mont2} where it is shown that they correspond to holomorphic
foliations with singularities with locally free rank 1 tangent sheaf. We now answer the question that he posed in
\cite{Suwa}, Remark 1.9, page 179 :

\vskip 2mm

\emph{The tangent sheaf of a dimension one reduced foliation on a complex $n$-dimensional manifold is a line bundle}

\vskip 2mm

It is a consequence of our Theorem \ref{main}, and of the fact that a reflexive sheaf of
rank one is a line bundle (see Lemma 1.1.15, page 154, in \cite{Okonek}).

We obtain an immediate proof of  Proposition 1.7 in \cite{Suwa}
(with no use of \cite{MY} or \cite{Suwa1}):

\vskip 2mm

\noindent \emph{Let ${\mathcal F}_{\varphi}$ be a holomorphic foliation with singularities.
\begin{enumerate}
\item If it is  reduced, the codimension of its singular locus is $\geq 2$;
\item if its tangent
sheaf ${\mathcal F}$  is locally free, and the
codimension of the singular locus of ${\mathcal F}_{\varphi}$ is at least 2, then the foliation is
reduced.
\end{enumerate}}
The first assertion follows from Proposition \ref{main}, and so does the second after noting that
locally free sheaves are reflexive.

\vskip 0.5cm

\section{Tangent sheaf vs. singular scheme}

From now on, all the foliations that we will consider will be defined in projective space, of codimension one, reduced,
and with singular set of
codimension $\geq 2$. In this section we characterize when the tangent sheaf is
locally free or split (i.e., direct sum of line bundles), in terms of the geometry of the singular scheme of the
foliation.

Recall that a degree $d$ codimension one holomorphic foliation with singularities is defined by a global section
$\omega \in H^0(\mathbb P^n, \Omega^1_{\mathbb P^n}(d+2))$. The form $\omega$ satisfies the integrability condition
$\omega \wedge d\omega=0$. The degree is the number of tangencies (counted with multiplicities) of a generic line with
the foliation. We can write $\omega = \sum F_j dz_j$, where the $F_j$ are homogeneous polynomials of the same degree
$\deg F_j=d+1$ satisfying $\sum_{j=0}^n z_j F_j=0$.

The singular set of the foliation is given by $F_0=\cdots =F_n=0$, and it has a natural structure of closed subscheme
of $\mathbb P^n$, given by the homogeneous ideal $(F_0, \ldots, F_n)$. Recall
(\cite{Hartshorne77}) that two
homogeneous ideal $I, J$ define the same projective scheme $X$ if and only if they have the same saturation, i.e.,
$I^{sat}=J^{sat}$, where
$$
I^{sat}=\bigcup_{l \geq 0} (I: (z_0,\dots,z_n)^l) = \bigoplus_{n} H^0(\PP^n, \mathcal I_X (n)),
$$
$\mathcal I_X$ being the ideal sheaf of the subscheme $X$.

The singular scheme can be also obtained (see \cite{Cukierman-Soares-Vainsencher}) as the closed subscheme of $\PP^n$
whose ideal sheaf is the image of the co-section
$$
\omega^*: (\Omega_{\PP^n}(d+2))^* \lra \Oc_{\PP^n}
$$
To simplify the notation, let us write $Z:=Sing(\F_{\varphi})$, and $\I_Z=\textrm{Im}(\omega^*)$ to denote the ideal
sheaf defining the singular subscheme in $\PP^n$. The induced map
$$
T\PP^n \lra \I_Z \otimes \Oc_{\PP^n}(d+2)
$$
is surjective, and moreover, it makes the following sequence exact:
\begin{equation}
\label{sucesionideal}
0 \longrightarrow {\mathcal F} \stackrel{\varphi}\longrightarrow T\mathbb P^n \longrightarrow \I_Z \otimes
\Oc_{\PP^n}(d+2) \longrightarrow 0.
\end{equation}
Observe that $\mathcal N_{\F}=\I_Z \otimes \Oc_{\PP^n}(d+2)$.
\\

We will focus on the case $n=3$. As we have pointed out, the
tangent sheaf ${\mathcal F}$ is reflexive. Its first Chern class
can be computed from (\ref{sucesionideal}) and equals $2-d$.

We will make use of a Theorem in \cite{Roggero-Valabrega}, that we state in the particular case
that will be of use for us. Let
$\epsilon=\frac{d}{2}-1$ when $d$ even, and $\epsilon=\frac{d-1}{2}-1$ for $d$ odd. Let $h^i({\mathcal
F}(\ell)):=\dim_{\mathbb C}H^i(\mathbb P^3, {\mathcal F}(\ell))$.

\begin{teorema}[\cite{Roggero-Valabrega}]\label{criterio}
The tangent sheaf ${\mathcal F}$ is locally free if and only if $h^2({\mathcal F}(p))=0$, for some $p\leq \varepsilon -
2$.

If $h^2({\mathcal F}(p))=0$  for $p=\varepsilon - 3$ when $d$ is even, or for $p\in \{\varepsilon -4, \varepsilon -3,
\varepsilon -2\}$ when $d$ is odd, then ${\mathcal F}$ splits.
\end{teorema}

 We now obtain some results relating algebro-geometric properties of the singular
scheme to the fact that the tangent sheaf is locally free, and
also to its being split. For us, a curve will be an
equidimensional, locally Cohen-Macaulay, dimension one subscheme
of $\PP^n$. Given a curve $C$ in $\PP^3$, the Hartshorne-Rao
module is defined to be
$$
\sum_{k \in \ZZ} H^1(\PP^3,\I_C(k)).
$$
It is well known that this module is finite dimensional as a $\CC$-vector space, and that is trivial if and only if $C$
is arithmetically Cohen-Macaulay.

First, we characterize those foliations having locally free tangent sheaf.
\begin{teorema}\label{locallyfree}
The tangent sheaf ${\mathcal F}$ is locally free if and only if the singular scheme $Z$ is a curve.
\end{teorema}
\begin{demostracion}
Suppose that the tangent sheaf is locally free, then $Z$ has no isolated points. The reason is that, in this case, the
singular scheme is given by the vanishing of the $2 \times 2$ minors of the $3 \times 2$ matrix corresponding to the
local expression of $\varphi$. Hence, $Z$ is locally determinantal and Cohen-Macaulay \cite{Eisenbud}. As the
codimension of $Z$ is at least two, we are done. Observe that the singular scheme has no embedded points and no
isolated points.

\vskip 3mm

To prove the converse, suppose that $Z$ satisfies the properties in the statement. We will make use of Theorem
\ref{criterio}.

By considering the short exact sequence (\ref{sucesionideal}), after tensoring with $\Oc_{\PP^3}(-q)$, and taking the
long exact sequence of cohomology we get:
$$
\begin{array}{rl}
\cdots \lra & H^1(\PP^3,\F(-q)) \lra H^1(\PP^3,T\PP^3(-q)) \lra H^1(\PP^3,\I_{Z}(d+2-q)) \lra \\
 \lra & H^2(\PP^3,\F(-q)) \lra  H^2(\PP^3,T\PP^3(-q))  \lra H^2(\PP^3,\I_{Z}(d+2-q)) \lra \cdots
\end{array}
$$
Observe that, from Bott's formula (see \cite{Okonek}), if we take $q>4$, then
$$
h^1(\PP^3,T\PP^3(-q))=h^2(\PP^3,T\PP^3(-q))=0 \, ,
$$
so
$$
 H^2(\PP^3,\F(-q))\simeq H^1(\PP^3,\I_{Z}(d+2-q)).
$$

Now, since the Harshorne-Rao module has finite dimension, we have
that
$$
H^1(\PP^3,\I_{Z}(d+2-q))=0
$$
for $q$ large enough, and hence for some $-q < \varepsilon-2$. Hence, $\F$ is locally free.
\end{demostracion}

\vskip 5mm

Now, we characterize foliations whose tangent sheaf splits.

\begin{teorema}\label{splitaCM}
Suppose $d >1$. $\F$ splits if and only if $Z$ is an arithmetically Cohen-Macaulay curve.
\end{teorema}

\begin{demostracion}
Suppose $\F$ splits. In particular, it is locally free, and hence $Z$ is a curve. Consider the exact sequence
$$
0 \longrightarrow {\mathcal F}(p) \stackrel{\varphi}\longrightarrow T\mathbb P^3(p) \longrightarrow
{\mathcal I}_{Z}(d+2+p) \longrightarrow 0 \label{sucesionfoliacion}
$$
and the cohomology exact sequence:
\begin{equation}\label{longcoho}
\dots \to H^1(\PP^3,T\PP^3(p))\to H^1(\PP^3,\I_Z(d+2+p)) \to H^2(\PP^3,\F(p)) \to \dots
\end{equation}

Bott's formula tells us that $H^1(\PP^3,T\PP^3(p))=0$ for every $p$. In addition, since $\F$ splits,
$H^2(\PP^3,\F(p))=0$ for every $p$, too. Then the Hartshorne-Rao module is trivial.

\vskip 3mm

Conversely, suppose $Z$ is an arithmetically Cohen-Macaulay curve.
We use Proposition \ref{criterio}. Take $p=\varepsilon-3$ if $d$
is even, and $p=\varepsilon-2$ if $d$ is odd. It will be
sufficient to show that
$$
H^2(\PP^3,\F(p))=0 .
$$
Consider the following  piece of the long exact cohomology sequence (\ref{longcoho}):
$$
\dots \to H^1(\PP^3,\I_Z(d+2+p)) \to H^2(\PP^3,\F(p)) \to H^2(\PP^3, T\PP^3(p)) \to \dots
$$
and note that since the Hartshorne-Rao module is trivial, we just need to check that
$$
H^2(\PP^3,T\PP^3(p))=0.
$$
Suppose first that $d$ is even; then, $\varepsilon-3=\frac{d}{2}-4$ and by Bott's formula,
$H^2(\PP^3,T\PP^3(\frac{d}{2}-4))=0$ for $d>0$.
If $d$ is odd, $\varepsilon-2=\frac{d-7}{2}$, and so, $H^2(\PP^3,T\PP^3(\frac{d-7}{2}))=0$ for
$d>-1$.
\end{demostracion}

It is well known that every arithmetically Cohen-Macaulay curve is connected, so we have:
\begin{corolario}
If the tangent sheaf $\F$ splits, then $Z$ is connected.
\end{corolario}

We now obtain  another characterization for the splitting of the tangent sheaf of a foliation:
\begin{teorema}\label{splitsaturado}
$\F$ splits if and only if it is locally free and the ideal $I=(F_0,F_1,F_2,F_3)$ is saturated.
\end{teorema}

\begin{proof}
From the exact sequence (\ref{sucesionideal}) and  Euler's sequence
$$
0 \lra \Oc_{\PP^3} \lra  \Oc_{\PP^3}^{\oplus 4} (1)  \lra T \PP^3 \lra 0
$$
we get the exact sequence:
\begin{equation}\label{mixta}
0 \longrightarrow {\mathcal F} \oplus \Oc_{\PP^3} \longrightarrow \Oc_{\PP^3}^{\oplus 4}(1)
\longrightarrow \I_Z \otimes \Oc_{\PP^3}(d+2) \longrightarrow 0 .
\end{equation}

If $\F$ splits, $H^1(\PP^3,{\mathcal F}(q) \oplus \Oc_{\PP^3}(q))=0$ for every $q$, and so we
have a surjective map:
\begin{equation}\label{surjective}
H^0(\PP^3, \Oc_{\PP^3}^{\oplus 4}(1+q)) \lra H^0(\PP^3,\I_Z \otimes \Oc_{\PP^3}(d+2+q))
\end{equation}
By simple inspection of the map inducing this surjection we conclude that
$$
H^0(\PP^3,\I_Z \otimes \Oc_{\PP^3}(d+2+q))=I_{d+2+q}
$$
for every $q\in \ZZ$, and therefore $I$ is saturated.

Conversely, if the ideal $I$ is saturated the map (\ref{surjective}) is surjective. Since
$$
H^1(\PP^3, \Oc_{\PP^3}^{\oplus 4}(1))=0,
$$
we get that $H^1(\PP^3,{\mathcal F}(q) \oplus \Oc_{\PP^3}(q))=0$,
and hence $H^1(\PP^3,{\mathcal F}(q))=0$, for every $q$.

In order to apply Horrocks' criterion for the splitting of the vector bundle
${\mathcal F}$ (see \cite{Okonek}, p. 39),
we just need to prove also that  $H^2(\PP^3,{\mathcal F}(q))=0$ for every $q$.
By Serre's duality
$$
H^2(\PP^3,{\mathcal F}(q))= H^1(\PP^3,{\mathcal F}^{*}(-q-4))\quad \text{for every}\quad q.
$$
But (this can be found in Hirzebruch \cite{Hirzebruch})
$$
{\mathcal F}^{*}\simeq \text{det}\,{\mathcal F}^{*} \otimes {\mathcal F},
$$
and so the vanishing follows from the previous calculation.
\end{proof}

\vskip 0.5cm

Let us observe that Theorem \ref{splitsaturado} translates into an algebraic setting the
splitting of the tangent sheaf of the foliation. Thus, it could be an effective tool in
trying to solve the following important open problem, posed in \cite{Nosotros}:
\emph{Given a foliation in $\PP^3$ with locally free tangent sheaf, does it split?}

The  question above can be restated as follows: let $\omega= \sum_{i=0}^3 F_i dz_i$ define
a degree $d$ foliation in $\PP^3$, with locally free tangent sheaf (i.e. the scheme
defined by $F_0=\cdots =F_3=0$ is a curve), is the ideal $(F_0, F_1, F_2, F_3)$ saturated?
Furthermore, is this ideal saturated for \emph{every reduced foliation}?

\vskip 0.5cm

We note that given a reduced foliation defined by $\omega= \sum_{i=0}^3 F_i dz_i$,
using for example the algorithm \emph{nsatiety( )}  implemented in the {\sc Singular} library
``noether.lib" \cite{Singular},
one immediately checks
whether the ideal $(F_0, F_1, F_2, F_3)$ is saturated or not.

\vskip 0.5cm

Let us point out that Theorem 3.3 in the paper of Campillo and
Olivares (\cite{Campillo-Olivares}) can be restated as follows:
for any foliation in $\PP^2$ defined by a 1-form
$$
\omega= \sum_{i=0}^2 F_i dx_i
$$
with singular set of codimension at least two, the ideal $(F_0,F_1,F_2)$ is
saturated. Our approach gives a much simpler proof of that result, since the tangent sheaf of
such a foliation is a line bundle.

\section{On the determination of a split foliation by the singular scheme}
Now we face the problem of deciding when a codimension one
foliation in $\PP^3$ is determined by its singular scheme,
as stated in the first Introduction. In
$\PP^2$, the problem was completely solved in
\cite{Campillo-Olivares}, Theorem 3.5:  a degree $r=0$, or $r \geq
2$ foliation is uniquely determined by its singular subscheme; for
degree $r=1$ they construct a 1-dimensional family of distinct
foliations with the same singular subscheme.

\vskip 0.3cm

In $\PP^3$, we study this problem just for foliations whose
tangent sheaf splits (recall that this is trivially true for foliations in the plane).
We develop a method to deal with the problem, and obtain
affirmative results for certain splitting types (Theorem \ref{caractporsing}, Theorem \ref{pull-back2}).
We also show examples for the remaining splitting types, showing  that the answer depends
on the particular singular scheme $Z$.

\vskip 0.3cm

We begin by noting the following
\begin{remark}\label{grado}
Two split foliations
defined by $F_0 dz_0 + F_1 dz_1 + F_2 dz_2 + F_3 dz_3$ and
$G_0 dz_0 + G_1 dz_1 + G_2 dz_2 + G_3 dz_3$,  with the same singular scheme $Z$ are of the same degree $d$.
Indeed, it is enough to observe that by Theorem \ref{splitsaturado}, the
homogenous ideal $(F_0, \ldots, F_3)=(G_0, \ldots, G_3)$ is saturated.
\end{remark}

We now prove a lemma that contains a basic idea for the problem. Recall (see, for example
\cite{Cukierman-Pereira}) that a codimension one holomorphic distribution is defined just
in the same way as a foliation, removing the integrability condition.

\begin{lema}\label{sicigia-lineal}
Let $\mathcal F_\varphi$  be a foliation defined by a projective
integrable 1-form $ \omega= \sum_{i=0}^3 F_i dz_i $ with split
tangent sheaf, and singular scheme $Z$.  Any distribution
$\mathcal F^{\prime}_{\varphi^{\prime}}$ with singular scheme $Z$
is split, and induces a linear syzygy
$$
\ell_0 F_0 + \ell_1 F_1 + \ell_2 F_2 + \ell_3 F_3=0 \, ,\quad
\ell_i\in \CC[z_0, z_1,z_2,z_3], \; \deg\ell_i=1\, , \; i=0,\ldots , 3.
$$
\end{lema}
\begin{demostracion}
First note that Theorem \ref{splitaCM} applies for distributions,
as the integrability does not enter in the proof; thus we conclude
that $\mathcal F^{\prime}$ splits. Then
$$
\mathcal F^{\prime} = \mathcal O_{\PP^3}(a^{\prime}) \oplus
\mathcal O_{\PP^3}(b^{\prime})\, , \quad \text{with}\quad
a^{\prime}+b^{\prime}=2-d, \; \text{and}\; a^{\prime}\leq b^{\prime}\leq 1.
$$
Note that  $a^{\prime}\leq b^{\prime}\leq 1$ follows from the stability of $T\PP^3$ (see \cite{Okonek}).
Furthermore, $a=a^{\prime}$ and $b=b^{\prime}$, as can be deduced by considering the long cohomology
sequences obtained from (\ref{sucesionideal}) and
\begin{equation}
0 \longrightarrow {\mathcal F^{\prime}}
\stackrel{\varphi^{\prime}}\longrightarrow T\mathbb P^3
\longrightarrow \I_Z \otimes \Oc_{\PP^n}(d+2) \longrightarrow 0,
\end{equation}
as we have that $h^0(\PP^3, \mathcal I_Z(\ell))$ equals
$$
h^0(\PP^3, T\PP^3(\ell -d -2))
- h^0(\PP^3, \mathcal F(\ell-d-2))\, ,
$$
and also
$$
h^0(\PP^3, T \PP^3(\ell -d-2))
- h^0(\PP^3, \mathcal F^{\prime}(\ell-d-2))\,
$$
for every $\ell \in \mathbb Z$.

The singular subscheme of the distribution ${\mathcal
F}^{\prime}_{{\varphi}^{\prime}}$ is $Z$, and it is defined by the
ideal $(F^{\prime}_0, F^{\prime}_1, F^{\prime}_2, F^{\prime}_3)$,
where the $F^{\prime}_i$ are degree $d+1$ homogeneous polynomials
such that $\omega^{\prime}= \sum_{i=0}^3 F^{\prime}_i dz_i$
defines ${\mathcal F}^{\prime}_{{\varphi}^{\prime}}$.

Hence, there is a $4\times 4$ matrix $M=(a_{ij})\in \text{GL}\,(4, \CC)$
such that
\begin{equation}\label{matrizactuando}
M \cdot \left( \begin{array}{c} F_0 \\
                                F_1 \\
                                F_2 \\
                                F_3 \end{array} \right)
= \left( \begin{array}{c}       F^{\prime}_0 \\
                                F^{\prime}_1 \\
                                F^{\prime}_2 \\
                                F^{\prime}_3 \end{array} \right) .
\end{equation}
Now, the Euler condition $\sum_{j=0}^3 z_j F^{\prime}_j =0$, gives
$$
\left( z_0 \; z_1 \; z_2 \; z_3 \right)
\cdot M
\cdot \left( \begin{array}{c} F_0 \\
                                F_1 \\
                                F_2 \\
                                F_3 \end{array} \right) =0,
$$
which produces a linear syzygy:
$$
\ell_0 F_0 + \ell_1 F_1 + \ell_2 F_2 + \ell_3 F_3=0.
$$
\end{demostracion}


\vskip 0.3cm

Now then, $H^0(\PP^3,\F \oplus \Oc_{\PP^3})$ is the vector space
of linear syzygies because of (\ref{mixta}). So such a
distribution as in the Lemma above gives a 1-dimensional linear
subspace of the vector space of linear syzygies
$H^0(\PP^3,\F\oplus \Oc_{\PP^3})$. Conversely, given a linear
subspace of $H^0(\PP^3,\F\oplus \Oc_{\PP^3})$ of dimension one, we
can take a generator and express it in the form
     $$
        \left( z_0 \; z_1 \; z_2 \; z_3 \right) \cdot M \, .
     $$
If $M\in GL(4, \CC)$, then equation $(\ref{matrizactuando})$ gives
coefficients for a homogeneous 1-form defining a distribution with
the same singular scheme. Therefore, we can assure that the family
of distributions with the same singular scheme as the foliation
$\F_\varphi$ is parameterized by a Zariski open subset ${\mathcal
D}_{\F_\varphi}$ of $\PP(H^0(\PP^3,\F\oplus \Oc_{\PP^3}))\subset
\PP(H^0(\PP^3,\Oc_{\PP^3}(1)^{\oplus 4}))$, obtained after
removing the  algebraic subset corresponding to non-invertible
matrices.

Therefore, \emph{foliations sharing the singular scheme $Z$ correspond to an algebraic subset of
${\mathcal D}_{\F_\varphi}$, defined by the equations expressing the
integrability condition}.

\vskip 0.3cm

Let us just point out that if we
find a basis of $H^0(\PP^3,\F\oplus \Oc_{\PP^3})$ it is easy to write explicit equations for
the integrability condition.

Now we present our first result on the characterization of a foliation by its singular
scheme:

\begin{teorema}\label{caractporsing}
Let $\mathcal F_\varphi$ be a degree $d$ reduced foliation in $\PP^3$, with $\mathcal F$ a rank two split
vector bundle, defined by a projective integrable one form
$
\omega= \sum_{i=0}^3 F_i dz_i .
$
Let $Z$ be its singular subscheme.
Suppose that
$$
\mathcal F = \mathcal O_{\PP^3}(a)\oplus \mathcal O_{\PP^3}(b)\, , \quad
\text{with}\; a\leq b\leq -1, \; a+b=2-d .
$$
Then,
if ${\mathcal F}^{\prime}_{{\varphi}^{\prime}}$
is another foliation in $\PP^3$ with the same singular subscheme $Z$,
defined by the form
$
\omega^{\prime}= \sum_{i=0}^3 F^{\prime}_i dz_i \, ,
$
we have $\omega=\lambda \omega^{\prime}$ for $\lambda \in \CC^*$.
\end{teorema}

\begin{demostracion}
By Lemma \ref{sicigia-lineal}, we
get a linear syzygy
$$
\ell_0 F_0 + \ell_1 F_1 + \ell_2 F_2 + \ell_3 F_3=0.
$$

From the sequence (\ref{mixta}), we conclude that this syzygy is a multiple of the
Euler relation
$\sum_{j=0}^3 z_j F_j =0$, so there is a $\lambda \in \CC^*$ with $M= \lambda \, Id$ and
also
$\omega^{\prime} = \lambda \omega$. \end{demostracion}

We analyze now the other possible splitting types. To start with we prove,
\begin{proposicion}\label{pull-back1}
Let $\mathcal F_\varphi$ be a degree $d$ reduced foliation in $\PP^3$.
$\mathcal F_\varphi$ is the linear pull-back of a degree $d$ reduced foliation in $\PP^2$ if and
only if
$
\mathcal F = \mathcal O_{\PP^3}(1)\oplus \mathcal O_{\PP^3}(1-d)\, .
$
\end{proposicion}
\begin{demostracion}
When $\mathcal F_\varphi$ is a linear pull-back it is known (see \cite{Cukierman-Pereira})
that the tangent sheaf splits, and that the splitting type is as in the statement of the theorem.
Indeed, note that $\mathcal O_{\PP^3}(1)$ is the tangent sheaf of the foliation whose leaves
are the fibers of the projection, and the splitting is given by the projection.

For the converse, suppose that the tangent sheaf of $\mathcal F_\varphi$  splits as
$$
\mathcal F = \mathcal O_{\PP^3}(1)\oplus \mathcal O_{\PP^3}(1-d)\, .
$$
Then, by first tensoring (\ref{mixta}) with $\mathcal O_{\PP^3}(-1)$  and
considering the exact cohomology sequence, we get four complex numbers $a_0, \ldots ,a_3$
(not all of them equal to zero)
such that $a_0 F_0 + a_1 F_1 + a_2 F_2 +a_3 F_3=0$. Suppose $a_0\neq 0$.

If we change coordinates:
$$
\left\{ \begin{array}{rcl}
z_0 & = & a_0 z^{\prime}_0 \\
z_1 & = & z^{\prime}_1 + a_1 z^{\prime}_0 \\
z_2 & = & z^{\prime}_2 + a_2 z^{\prime}_0 \\
z_3 & = & z^{\prime}_3 + a_3 z^{\prime}_0 \end{array}\right.
$$
we get a new expression for the form defining $\mathcal F_\varphi$:
$$
\eta = G_1 dz^{\prime}_1 + G_2 dz^{\prime}_2 + G_3 dz^{\prime}_3
$$
where $G_1$, $G_2$ and $G_3$ are homogeneous in the new variables.

Let us prove that $G_i \in \CC [z^{\prime}_1,
z^{\prime}_2,z^{\prime}_3]$. To do that, we go to the affine open
$z^{\prime}_3=1$, and get a 1-form
$$
\eta_3 = G_1(z^{\prime}_0,z^{\prime}_1,z^{\prime}_2,1) dz^{\prime}_1 +
G_2(z^{\prime}_0,z^{\prime}_1,z^{\prime}_2,1) dz^{\prime}_2
$$
that defines the foliation in that affine chart.
To simplify, we write:
$$
f(z^{\prime}_0,z^{\prime}_1,z^{\prime}_2)=G_1(z^{\prime}_0,z^{\prime}_1,z^{\prime}_2,1),
\quad
g(z^{\prime}_0,z^{\prime}_1,z^{\prime}_2)=G_2(z^{\prime}_0,z^{\prime}_1,z^{\prime}_2,1).
$$
As the singular set of the foliation is of codimension two, $f$ and $g$ have no common
irreducible factor.

 Note that $\eta_3$ is integrable (i.e. $\eta_3 \wedge d\eta_3=0$), and hence we have:
$$
\frac{\partial f}{\partial z^{\prime}_0} \;  g =
\frac{\partial g}{\partial z^{\prime}_0} \; f \, .
$$
From the equality above, and as $f$ and $g$ have no common irreducible factor,
it follows that $f$  divides $\frac{\partial f}{\partial z^{\prime}_0}$ and hence
$\frac{\partial f}{\partial z^{\prime}_0} =0$. Analogously,
$\frac{\partial g}{\partial z^{\prime}_0}=0$.
Thus $f$ and $g$ do
not depend on $z^{\prime}_0$.
Proceeding in a analogous way in
another affine open set, we can deduce that $G_0, G_1$ and $G_2$
are in $\CC [z^{\prime}_1, z^{\prime}_2,z^{\prime}_3]$.

We can finally conclude that $\mathcal F_\varphi$ is a linear pull-back of the
foliation given by
$$
\eta = G_1 dz^{\prime}_1 + G_2 dz^{\prime}_2 + G_3 dz^{\prime}_3
$$
in the projective plane $z^{\prime}_0=0$.
\end{demostracion}

\vskip 0.3cm

Now, we come back to the problem of deciding whether the singular scheme
characterizes split foliations and prove:
\begin{teorema}\label{pull-back2}
Let $\mathcal F_\varphi$ be a degree $d\neq 1$ reduced foliation in $\PP^3$, with
$$
\mathcal F = \mathcal O_{\PP^3}(1)\oplus \mathcal O_{\PP^3}(1-d)\, ,
$$
and let $Z$ be its singular scheme. There is no other foliation
with singular scheme $Z$.
\end{teorema}

\begin{demostracion}
Suppose there is another foliation $\mathcal G_\psi$ with singular scheme $Z$. As $Z$ is an arithmetically
Cohen-Macaulay curve, its tangent sheaf $\mathcal G$ splits, with the same splitting type as that of
$\mathcal F$.
Hence, from the previous Proposition \ref{pull-back1} both
$\mathcal F_\varphi$ and $\mathcal G_\psi$ are linear pull-backs of two foliations of a projective plane in $\PP^3$.

As $Z$ is a cone, taking any plane $H\subset \PP^3$ not passing through the vertex $p$ of $Z$, we obtain
both $\mathcal F_\varphi$ and $\mathcal G_\psi$ as linear pull-backs from $p$ of two degree $d$ foliations
in $H$ (see for example Section 2 in \cite{Cerveau-Lins-Neto}). Now, as $d\neq 1$, by Theorem 3.5 in
\cite{Campillo-Olivares} these two foliations are the same,
as their common singular subscheme is $H\cap Z$.
\end{demostracion}

\vskip 0.5cm

Finally, after Theorems \ref{caractporsing} and \ref{pull-back2}, there is just one
splitting type for each degree $d$ that remains to be considered:
$$
\mathcal F = \mathcal O_{\PP^3} \oplus \mathcal O_{\PP^3}(2-d)\, .
$$
We can make the following remarks:

\noindent $\bullet$ If $d=1$, a foliation with splitting type
$
\mathcal F = \mathcal O_{\PP^3} \oplus \mathcal O_{\PP^3}(1)
$
is a linear pull-back (by Proposition \ref{pull-back1}). The intersection with a
general plane $H$ gives a degree $1$ foliation in $H$ whose singular scheme
determines $Z$. However (see Remark 3.6 in \cite{Campillo-Olivares}), there is a one
dimensional family of  different
foliations in the plane $H$ with the
same singular scheme, giving different foliations in $\PP^3$ by pull-back
\emph{with the same singular scheme as $\mathcal F_\varphi$}. In fact, for any
degree one foliation in $\PP^2$, using the immediate translation for dimension 2
of Lemma \ref{sicigia-lineal} we conclude that the family of plane foliations
with the same singular scheme is of dimension one.

\vskip 0.3cm

\noindent $\bullet$ If $d=2$, for a foliation ${\mathcal G}_\psi$
with split tangent sheaf $ \mathcal G = \mathcal O_{\PP^3} \oplus
\mathcal O_{\PP^3} $ there is a pair of linear vector fields $X$,
and $Y$ in $\psi(\mathcal G)$  generating a Lie algebra. According
to the classification, this can be abelian or isomorphic to the
affine Lie algebra.

In the abelian case,
we can diagonalize simultaneously both vector fields, and the resulting
foliation is logarithmic of type $\mathcal L (1,1,1,1)$,
as can be seen in \cite{Cukierman-Pereira}. Note that the other
possible type of degree 2 logarithmic
foliations $\mathcal L(1,1,2)$ have a tangent sheaf which is not locally free,
as it contains isolated points
(see \cite{Cukierman-Soares-Vainsencher}).

A foliation in $\mathcal L (1,1,1,1)$  is defined by a form (see \cite{Omegar}):
$$
\omega= \ell_0 \ell_1 \ell_2 \ell_3 \sum_{i=0}^3 \lambda_i \frac{d \ell_i}{\ell_i},
$$
where  the $\ell_i$ are linear forms in general position, and the
scalars satisfy
\begin{equation}\label{log}
\lambda_i\in \CC^* \, , \quad
\sum_{i=0}^3 \lambda_i =0.
\end{equation}

Its singular scheme $Z$ is the given by six lines
giving the edges of a tetrahedron obtaining by intersecting any two of the
$\ell_i$.

Now, the same linear forms $\ell_i$ with different choices of scalars
satisfying (\ref{log}) provide distinct foliations
with the same singular scheme, as these determine the holonomy of the foliation.

\vskip 3mm

If the Lie algebra generated by the linear vector fields is isomorphic to the
affine Lie algebra, after a linear change of coordinates we can choose generators
$X^{\prime}$ and $Y^{\prime}$ with
$$
[X^{\prime}, Y^{\prime}]=Y^{\prime} \, .
$$
Thus, we are in the setting of the exceptional component of degree 2 foliations in $\PP^3$,
introduced by Cerveau and Lins Neto
(see \cite{Cerveau-Lins-Neto}): a general
member is given by a split foliation (\cite{Nosotros}), with tangent sheaf
$
\mathcal F = \mathcal O_{\PP^3} \oplus \mathcal O_{\PP^3},
$
and they prove that this foliation ${\mathcal F}_\varphi$  is rigid
(i.e. an open dense
subset of the
component is described by the action of the group $PGL(4, \CC)$ on that foliation).

The singular scheme $Z$ has three irreducible components: a line $\ell$, a conic $C$ tangent to $\ell$ at
a point $p$, and a twisted cubic with the line $\ell$ as an inflection line at $p$.

In this setting,  the reader can check
that the equations expressing the integrability, obtained
from  the basis of
$H^0(\PP^3, \F \oplus \mathcal O_{\PP^3})$ given by $X^{\prime}, Y^{\prime}$
(see their explicit expression in \cite{Cerveau-Lins-Neto}) and the radial vector field
$R$ have a unique solution: the one corresponding to the foliation $\F_{\varphi}$ itself.

\vskip 0.3cm

\noindent $\bullet$ If $d >2$, we have that $\dim H^0(\PP^3, \F \oplus \mathcal O_{\PP^3})=2$.
In \cite{Nosotros}, it is proven that there is an exceptional component of the
space of (codimension one) degree $d$ foliations in $\PP^3$, whose general member is a foliation
$\F_{\varphi}$ with split
tangent sheaf of the form
$
\mathcal O_{\PP^3} \oplus \mathcal O_{\PP^3}(2-d)
$.

$\F_{\varphi}$ is associated to a representation of the affine Lie algebra, in such a way
that in an affine open set, $\F_{\varphi}$ is defined by a one form
$
\omega = i_S \, i_X (d\, Vol)\, ,
$
where $S$ is a linear vector field and $X$ is quasi-homogeneous:
$$
S= (1+d+d^2) z_1 \frac{\partial}{\partial z_1} + (1+d) z_2 \frac{\partial}{\partial z_2}
+ z_3 \frac{\partial}{\partial z_3},
$$
$$
X=(1+d+d^2) z_2^d \frac{\partial}{\partial z_1} + (1+d) z_3^d \frac{\partial}{\partial z_2}
+ \frac{\partial}{\partial z_3}.
$$
Thus, we
can take $R$ and $S$ as a basis of the vector space
$H^0(\PP^3, \F \oplus \mathcal O_{\PP^3})$.

The vector field $S$ corresponds to a
linear syzygy that can be interpreted as a matrix $M\in GL(4, \CC)$
acting on the coefficients of the projective form
$
\overline{\omega}= \sum_{i=0}^{3} F_i dz_i ,
$  extending
$$
\omega = (1+d)(z_2 - z_3^{d+1}) dz_1 - (1+d+d^2)(z_1 - z_2^dz_3)dz_2 -
(1+2d+2d^2+d^3)(z_2^{d+1}-z_1z_3^{d})dz_3
$$
to projective space,
and defining $\F_{\varphi}$. By direct computation, we note that
the form $\omega_1$ with coefficients given by the equation $(\ref{matrizactuando})$
is not integrable.

Hence, any distribution with the same singular scheme $Z$ as the foliation $\F_{\varphi}$
(see \cite{Nosotros} for its explicit geometric description)
can be defined by a form
$\alpha_0 \omega_0 + \alpha_1 \omega_1$, where $\omega_0=\overline{\omega}$.
The integrability condition is
$$
\alpha_0 \alpha_1 \omega_0 \wedge d\omega_1 + \alpha_0 \alpha_1 \omega_1 \wedge d\omega_0
+ \alpha_1^2 \omega_1 \wedge d\omega_1 \, .
$$
An explicit computation shows that the only solution is $\alpha_1=0$, and so the only
foliation with singular scheme $Z$ is $\F_{\varphi}$.

\vskip 0.5cm

In principle, we do not know whether there are other irreducible
components of the space of foliations with the same split tangent
sheaf. In fact, despite the recent important results in the
literature,  the knowledge of these spaces is still quite
uncomplete.

\vskip 0.5cm

\begin{remark}
Observe that the proofs of the results in  Sections 2 and 3 except those of
Proposition \ref{pull-back1} and Theorem \ref{pull-back2},
could be immediately adapted to deal with singular distributions, since
no use is made of the integrability condition.

Note also that we can
extend our approach  to foliations in  $\PP^n$, and get immediate
results: Theorem \ref{splitsaturado}, Lemma \ref{sicigia-lineal}
and Theorem \ref{caractporsing} are valid in that setting.
\end{remark}

\end{document}